\documentclass[a4paper,10pt]{amsart}
 \usepackage{amsfonts,mathrsfs}
 \usepackage{amsthm}
 \usepackage{amssymb}
 \usepackage{enumerate}
 \usepackage{tikz-cd}
  \usepackage{hyperref}
\usepackage{cleveref}
  \theoremstyle{definition}
\newtheorem{Def}{Definition}[section]
\newtheorem{ex}[Def]{Example}
\newtheorem{rem}[Def]{Remark}

\theoremstyle{plain}
\newtheorem{prop}[Def]{Proposition}
\newtheorem{thm}[Def]{Theorem}
\newtheorem*{thm*}{Theorem}
\newtheorem{lem}[Def]{Lemma}
\newtheorem{cor}[Def]{Corollary}
\newtheorem*{cor*}{Corollary}
\newtheorem{con}[Def]{Conjecture}
\newtheorem*{con*}{Conjecture}
\newtheorem{frag}[Def]{Question}
\newtheorem*{verm*}{Vermutung}

\usepackage{mathrsfs}
\newcommand{\Par}{\operatorname{par}}

\newcommand{\arf}{\operatorname{arf}}
\newcommand{\Pic}{\operatorname{Pic}}

\newcommand{\id}{\operatorname{id}}

\newcommand{\shHom}{\operatorname{{\mathscr Hom}}} %Sheaf Hom
\newcommand{\Spec}{\operatorname{Spec}}

\newcommand{\T}{\operatorname{T}}

\newcommand{\divv}{\operatorname{div}}

\newcommand{\Gal}{\operatorname{Gal}}
\newcommand{\Char}{\operatorname{char}}

\newcommand{\Tr}{\operatorname{Tr}}
\newcommand{\J}{\operatorname{J}}

\newcommand{\tr}{\operatorname{tr}}

\newcommand{\GW}{\operatorname{GW}}

\newcommand{\cO}{{\mathcal O}}

\newcommand{\cT}{{\mathcal T}}

\newcommand{\A}{{\mathbb A}}

\newcommand{\C}{{\mathbb C}}
\newcommand{\R}{{\mathbb R}}
\newcommand{\pp}{\mathbb{P}}
\newcommand{\Q}{{\mathbb Q}}
\newcommand{\N}{{\mathbb N}}
\newcommand{\Z}{{\mathbb Z}}

\newcommand{\sg}{\operatorname{sg}}
\title[A signed count of 2-torsion points]{A signed count of 2-torsion points on real abelian varieties}
\author{Mario Kummer}
\address{Technische Universit\"at, Dresden, Germany} 
\email{mario.kummer@tu-dresden.de}
\thanks{The author has been supported by the DFG under Grant No. 502861109}
\makeatletter
\newcommand{\bigperp}{%
  \mathop{\mathpalette\bigp@rp\relax}%
  \displaylimits
}

\newcommand{\bigp@rp}[2]{%
  \vcenter{
    \m@th\hbox{\scalebox{\ifx#1\displaystyle2.1\else1.5\fi}{$#1\perp$}}
  }%
}
\makeatother

\newcommand{\comment}[1]{}

\begin{document}


\begin{abstract}
 We prove that a natural signed count of the $2$-torsion points on a real principally polarized abelian variety $A$ always equals to $2^{g}$ where $g$ is the dimension of $A$. When $A$ is the Jacobian of a real curve we derive signed counts of real odd theta characteristics. These can be interpreted in terms of the extrinsic geometry of contact hyperplanes to the canonical embedding of the curve. We also formulate a conjectural generalization to arbitrary fields in terms of $\mathbb{A}^1$-enumerative geometry.
\end{abstract}
\maketitle

\section{Introduction}
Enumerative geometry classically counts the number of solutions to geometric questions over the complex numbers or, more generally, over an algebraically closed field. While over $\C$ the number of solutions usually is constant for generic instances of the problem, the number of real solutions to a problem posed over $\R$ can vary. One approach to deal with this phenomenon is to attach to each solution an instrinsic sign $\pm1$ such that the sum of these signs over all real solutions does not depend on the specific instance of the enumerative problem anymore. Ideally, such a signed count gives a sharp lower bound on the total number of real solutions. In order to illustrate this principle consider the problem of counting lines on a smooth cubic hypersurface in projective three-space. While over $\C$ the answer is always 27 \cite{cayley1849triple}, there are real smooth cubic surfaces containing 3, 7, 15 or 27 real lines \cite{schlafli1858attempt}. However, Segre \cite{segre} introduced  an instrinsic way to attach a sign $\pm1$ to each real line on a given real cubic surface in way that the sum of these signs always equals to three. Note that this result implies in particular that there must always be at least three lines. As there actually are real smooth cubic surfaces with only three real lines, this is a sharp lower bound. These results where recently generalized and popularized in \cite{FK} and \cite{OK}.
In this article, we introduce such a signed count for the number of real $2$-torsion points on a principally polarized real abelian variety and we prove that it always equals to $2^g$ where $g$ is the dimension of the abelian variety (\Cref{thm:signed2tor}). This is the smallest number of real $2$-torsion points a real abelian variety of dimension $g$ can have. While previously signed counts are mostly realized as the Euler number of a well chosen vector bundle on a parameter space of the enumerative problem, we prove our signed count in situs on the abelian variety. To that end we introduce for every $m\in\N$ and every field $k$ a bilinear pairing on $A_m(k)\times A^\vee_m(k)$ valued in $k^\times/(k^\times)^m$, where $A^\vee$ is the abelian variety dual to the abelian $A$ over $k$ and $A_m(k)$ the group of $k$-rational $m$-torsion points (\Cref{sec:pairing}). In the case $m=2$ and $k=\R$ we can identify $\R^\times/(\R^\times)^2$ with $\{\pm1\}$ and the sign attached to a real $2$-torsion point is then given by the pairing of the point with its image under the principal polarization. By applying these results to the Jacobian variety of a real curve and the principal polarization given by a theta divisor (\Cref{sec:realjac}), we also derive a signed count of real odd theta characteristics (\Cref{thm:oddcount}). This in turn implies signed counts of contact hyperplanes to canonically embedded curves where the signs can be interpreted in terms of the extrinsic geometry of the canonical curves (\Cref{thm:cont1} and its corollaries). In particular, \Cref{cor:cont3} generalizes a signed count of bitangents to plane quartic curves from \cite{larsonvogt}.
The concept of a signed count can be generalized to arbitrary fields $k$ as pioneered in \cite{levasp,kasswickel} where the count takes its value in the Grothendieck--Witt group $\GW(k)$ of $k$. In \Cref{sec:arith} we formulate a conjectural generalization of our signed count to this context and prove it in a few special cases.

\subsection{Conventions}
If $k$ is a field we denote its algebraic closure by $\overline{k}$. If $K$ is a field extension of $k$ and $X$ is a scheme over $k$, we denote by $X(K)$ the set of $K$-rational points of $X$. The base change of $X$ to $K$ is denoted by $X_K$. A variety over $k$ is a separated scheme of finite type over $k$ which is geometrically integral.

\section{Preliminaries on abelian varieties}
In this section we recall some facts on abelian varieties. We mostly stick to the notation from \cite{milneabel}.
Let $A$ be an abelian variety of dimension $g$ over the field $k$.
For any natural number $m\in\N$ not divisible by $\Char(k)$ the morphism $A\to A$ defined by multiplication by $m$ is \'etale of degree $m^{2g}$ \cite[Theorem~8.2]{milneabel}.
We denote by $A_m$ the fiber over the natural element under this morphism, i.e. $A_m$ is the scheme of $m$-torsion points.
We denote by $A^\vee$ the dual abelian variety, see \cite[\S9, \S10]{milneabel}.
\subsection{The Weil pairing}
Let $K/k$ be a field extension, let $m\in\N$ a natural number not divisible by $\Char(k)$ and $\mu_m(K)\subset K^\times$ the group of $m$th roots of unity in $K$. There is a natural bilinear pairing
\begin{equation*}
 e_m\colon  A_m(K)\times A_m^\vee(K)\to\mu_m(K)
\end{equation*}
called the \emph{Weil pairing}. If $K$ is algebraically closed, then $e_m$ is a perfect pairing. The Weil pairing $e_m$ has several equivalent descriptions, see e.g. \cite[\S16]{milneabel}. We will use the following one (see \cite[VI, \S4, Theorem 12]{langabel} for a proof that it agrees with the more standard definition given in \cite[\S16]{milneabel}).
Let $a\in A_m(K)$ and $a'\in A^\vee_m(K)\subset\Pic(A_K)$. Let $a'$ be represented by the divisor $D$ on $A_K$ and assume that the support of $D$ neither contains $a$ nor $0$. Since $A^{\vee\vee}=A$ \cite[\S9]{milneabel}, we can also consider $a$ as an element of $\Pic(A^\vee_K)$ represented by a divisor $D'$ on $A^\vee_K$ whose support neither contains $a'$ nor $0$. There are rational functions $f$ on $A_K$ and $f'$ on $A_K^\vee$  such that $\divv(f)=m\cdot D$ and $\divv(f')=m\cdot D'$.
Then one has
\begin{equation}
 e_m(a,a')=\frac{f'(0)\cdot f(a)}{f(0)\cdot f'(a')}\in\mu_m(K).\label{eq:weilextra}
\end{equation}
If $K/k$ is a Galois extension and $\sigma\in\Gal(K/k)$, then $\sigma(e_m(a,a'))=e_m(\sigma(a),\sigma(a'))$.
If $\lambda\colon  A\to A^\vee$ is a polarization, see \cite[\S13]{milneabel}, then the induced bilinear form 
\begin{equation*}
 A_m(K)\times A_m(K)\to\mu_m(K),\, (a,a')\mapsto e_m(a,\lambda (a'))
\end{equation*} 
is skew-symmetric \cite[Proposition~6.16]{milneabel}.

\section{Another pairing on $m$-torsion points}\label{sec:pairing}
We fix $m\in\N$ which is not divisible by the characteristic of the field $k$. We let $A$ be an abelian variety of dimension $g$ over $k$.
The goal of this section is to define a bilinear pairing on $A_m(k)\times A^\vee_m(k)$ which shares several properties with the Weil pairing $e_m$. However, its target will be the cokernel $k^{\times}/ (k^\times)^m$ rather than the kernel $\mu_m(k)$ of the $m$th power map $k^\times\to k^\times$.

Let $a\in A_m(k)$ and $a'\in A^\vee_m(k)$. Let $a'\in\Pic(A)$ be represented by the divisor $D$ on $A$ whose support neither contains $a$ nor $0$. Since $a'$ is $m$-torsion, there is a rational function $f$ on $A$ such that $\divv(f)=m\cdot D$. We define 
\begin{equation*}
 b_m(a,a'):= \frac{f(a)}{f(0)}\in k^{\times}/ (k^\times)^m.
\end{equation*}
This is well-defined since the rational function $\tilde{f}$ obtained from another choice $\tilde{D}$ of a divisor representing $a'$ would be of the form $\tilde{f}=f\cdot h^m$ where $\divv(h)=D'-D$. It follows directly from the construction that $b_m$ is linear in the second factor:

\begin{lem}\label{lem:linear}
 For all $a\in A_m(k)$ and $a_1',a_2'\in A_m^\vee(k)$ we have \begin{equation*}
 b_m(a,a_1'+a_2')=b_m(a,a_1')\cdot b_m(a,a_2').
\end{equation*}
\end{lem}

By carrying out the same construction for $A^\vee$ instead of $A$, we obtain a pairing 
\begin{equation*}
 b_m'\colon A_m^\vee(k)\times A_m(k)\to k^{\times}/ (k^\times)^m
\end{equation*}
since $A=(A^\vee)^\vee$. The two pairings $b_m$ and $b_m'$ are related via the Weil pairing. Indeed, let $a\in A_m(k)$ and $a'\in A_m^\vee(k)$. Let $a'$ be represented by a divisor $D$ on $A$ whose support neither contains $a$ nor $0$. Similarly, let $D'$ be a divisor on $A^\vee$ representing $a$ and whose support neither contains $b$ nor $0$. Let $f$ and $f'$ be rational functions on $A$ and $A^\vee$ respectively such that $\divv(f)=m\cdot D$ and $\divv(f')=m\cdot D'$. Then we have
\begin{equation}\label{eq:weilextra2}
 \frac{f(a)}{f(0)}=e_m(a,a')\cdot \frac{f'(a')}{f'(0)}
\end{equation}
 by the definition of the Weil pairing $e_m$ that we gave in \Cref{eq:weilextra}. In particular, \Cref{eq:weilextra2} implies the identity 
 \begin{equation}\label{eq:weilbm}
  b_m(a,a')=e_m(a,a')\cdot b_m'(a',a).
 \end{equation}
 The following theorem summarizes the properties of the pairing $b_m$.

\begin{thm}\label{thm:pairing}
 The pairing $$b_m\colon A_m(k)\times A_m^\vee(k)\to k^{\times}/ (k^\times)^m$$ satisfies the following for all $a,a_1,a_2\in A_m(k)$ and $a',a_1',a_2'\in A_m^\vee(k)$:
 \begin{enumerate}[(i)]
  \item $b_m(a,a_1'+a_2')=b_m(a,a_1')\cdot b_m(a,a_2')$;
  \item $b_m(a_1+a_2,a')=b_m(a_1,a')\cdot b_m(a_2,a')$;
  \item $b_m(a,a')=e_m(a,a')\cdot b_m'(a',a)$ where $e_m$ is the Weil pairing.
  %\item If $m=m_1\cdot m_2$ and $a\in A_{m_1}(k)$, then $b_{m}(a, a')=b_{m_1}(a,m_2\cdot a')$ where both sides are interpreted as elements of $k^{\times}/ (k^\times)^{m_1}$.
 \end{enumerate}
\end{thm}

\begin{proof}
 $(i)$ is \Cref{lem:linear} and $(iii)$ was observed in \Cref{eq:weilbm}. Then $(ii)$ follows from $(i)$, $(iii)$ and the bilinearity of the Weil pairing:
 \begin{align*}
  b_m(a_1+a_2,a')&=e_m(a_1+a_2,a')\cdot  b_m'(a',a_1+a_2)\\
  &= e_m(a_1,a')\cdot b_m'(a',a_1)\cdot e_m(a_2,a')\cdot b_m'(a',a_2)\\
  &=b_m(a_1,a')\cdot b_m(a_2,a'). \qedhere
 \end{align*}
\end{proof}

The pairing $b_m$ is also compatible with principal polarizations.

\begin{lem}\label{lem:princpair}
 Let $\lambda\colon A\to A^\vee$ a principal polarization and $a_1,a_2\in A_m(k)$. Then
 \begin{equation*}b_m(a_1,\lambda(a_2))=b_m'(\lambda(a_1),a_2).\end{equation*}
\end{lem}

\begin{proof}
 Let $\iota\colon A\to(A^\vee)^\vee$ the canonical isomorphism and $D'$ a divisor on $A^\vee$ that represents $\iota(a_2)$ whose support does not contain $0$ or $\lambda(a_1)$. The homomorphism $\lambda\colon A\to A^\vee$ induces a homomorphism $\lambda^\vee\colon (A^\vee)^\vee\to A^\vee$ which corresponds to the pull back of degree zero line bundles on $A^\vee$ to $A$ via $\lambda$. Thus $\lambda^\vee(\iota(a_2))$ is represented by $D:=\lambda^{-1}(D')$. Since $\lambda$ is a polarization, we have $\lambda=\lambda^\vee \circ\iota$. Thus $\lambda(a_2)=\lambda^\vee(\iota(a_2))$ is represented by the divisor $D$. 
 Let $f$ be a rational function on $A$ such that $\divv(f)=m\cdot D=m\cdot \lambda^{-1}(D')$ and let $\mu\colon A^\vee\to A$ the inverse of $\lambda$. Then we have \begin{equation*}\frac{f(a_1)}{f(0)}=\frac{(f\circ\mu)(\lambda(a_1))}{(f\circ\mu)(0)}.\end{equation*}
 The left-hand side is $b_m(a_1,\lambda(a_2))$ by definition.
 Because $\divv(f\circ\mu)=m\cdot D'$, the right-hand side is $b_m'(\lambda(a_1),a_2)$.
\end{proof}

\begin{cor}\label{cor:bande}
 Let $\lambda\colon A\to A^\vee$ a principal polarization and $a_1,a_2\in A_m(k)$. Then $$b_m(a_1,\lambda(a_2))=e_m(a_1,\lambda(a_2))\cdot b_m(a_2,\lambda(a_1)).$$
\end{cor}

\begin{proof}
 We have $b_m(a_2,\lambda(a_1))=b_m'(\lambda(a_2),a_1)$ by \Cref{lem:princpair} and 
 \begin{equation*}
  e_m(a_1,\lambda(a_2))\cdot b_m'(\lambda(a_2),a_1)=b_m(a_1,\lambda(a_2))
 \end{equation*}
 by \Cref{thm:pairing}(iii).
\end{proof}

For the rest of the section we focus on the case $m=2$.

\begin{Def}\label{def:q2form}
 Let $\lambda\colon A\to A^\vee$ a principal polarization. We define
 \begin{equation*}
  q_2\colon A_2(k)\to k^{\times}/ (k^\times)^2,\, a\mapsto b_2(a,\lambda(a)).
 \end{equation*}
\end{Def}

\begin{prop}\label{thm:q2prop}
 Let $\lambda\colon A\to A^\vee$ a principal polarization.
 For all $a_1,a_2\in A_2(k)$ we have $q_2(a_1+a_2)=e_2(a_1,\lambda(a_2))\cdot q_2(a_1)\cdot q_2(a_2)$.
\end{prop}

\begin{proof}
 We have by (i) and (ii) in \Cref{thm:pairing}: 
 \begin{equation*}
 q_2(a_1+a_2)=q_2(a_1)\cdot q_2(a_2)\cdot b_2(a_1,\lambda(a_2))\cdot b_2(a_2,\lambda(a_1)). 
 \end{equation*}
 Using \Cref{cor:bande} we now obtain:
 \begin{equation*}
 q_2(a_1)\cdot q_2(a_2)\cdot b_2(a_1,\lambda(a_2))\cdot b_2(a_2,\lambda(a_1))=q_2(a_1)\cdot q_2(a_2)\cdot e_2(a_1,\lambda(a_2))\cdot b_2(a_2,\lambda(a_1))^2. 
 \end{equation*} 
This implies the claim.
\end{proof}

\begin{ex}
 Consider the elliptic curve $(E,O)$ defined over $\Q$ by
 \begin{equation*}
  y^2=x^3-x.
 \end{equation*}
 The three nontrivial $2$-torsion points of $E$ are $P_1=(-1,0)$, $P_2=(0,0)$ and $P_3=(1,0)$. We consider the standard principal polarization $\lambda$ that sends a point $P$ to the divisor class of $P-O$. A basis of $E_2(\Q)$ considered as vector space over the field of two elements is given by $P_1,P_2$. With respect to this basis the bilinear forms $b_2(-,\lambda(-))$ and $e_2(-,\lambda(-))$ are represented by the following two matrices:
 \begin{equation*}
  \begin{pmatrix}
   2& -1\\
   1& -1
  \end{pmatrix} \textnormal{ and }
  \begin{pmatrix}
   1& -1\\
   -1& 1
  \end{pmatrix}.
 \end{equation*}
 The entries of the left matrix are considered to be square classes and the entries of the right matrix are from $\mu_2(\Q)=\{\pm1\}$. Finally, we have $q_2(P_1)=2$, $q_2(P_2)=-1$ and $q_2(P_3)=2$.
\end{ex}

\subsection{Jacobians of curves}
Let $X$ be a smooth projective geometrically irreducible curve of genus $g$ over $k$ with $P\in X(k)$. Let $J$ be the Jacobian variety of $X$. 
As a reference for the basic definitions and statements used in the following, we recommend \cite{MR861976}.
We denote by $X^{(n)}$ the $n$th symmetric product of $X$ and we identify elements of $X^{(n)}$ with effective divisors of degree $n$ on $X$. The morphism
\begin{equation*}
\pi\colon X^{(g)}\to J,\, D\mapsto [D-gP]
\end{equation*}
is surjective and birational. The image of $X^{(g-1)}\to X^{(g)},\, D\mapsto P+D$ composed with $\pi$ is a divisor $\Theta$ on $J$ called the \emph{theta divisor}. The map 
\begin{equation}\label{eq:princpolthet}
\lambda\colon J\to J^\vee,\, a\mapsto [t^*_a\Theta-\Theta]
\end{equation}
is a principal polarization. Here $t_a\colon J\to J$ is translation by $a$. The goal of this subsection is to describe the pairing $b_m(-,\lambda(-))$ solely in terms of the curve $X$. To this end let $a_1,a_2\in J_m(k)$. Let $a_1$ be represented by a divisor $D_1=\sum_{i=1}^g(Q_i-P)$ with $Q_i\in X(\bar{k})$ which is invariant under the action of the absolute Galois group of $k$. Similarly, let $a_2$ be represented by the divisor $D_2=\sum_{i=1}^r(B_i-A_i)$ with $A_i,B_i\in X(\bar{k})\setminus\{P,Q_1,\ldots,Q_g\}$.

\begin{lem}
 The element $\lambda(a_2)\in\Pic(J)$ is represented by the divisor
 \begin{equation*}
  \bar{D}_2:=\sum_{i=1}^r(t^*_{P-A_i}\Theta-t^*_{P-B_i}\Theta)
 \end{equation*}
 on $J$.
\end{lem}

\begin{proof}
 This follows from the theorem of the square \cite[Theorem~6.7]{milneabel}:
 \begin{eqnarray*}
  [\sum_{i=1}^r(t^*_{P-A_i}\Theta-t^*_{P-B_i}\Theta)]&=& [t^*_{\Sigma_{i=1}^r(P-A_i)}\Theta-t^*_{\Sigma_{i=1}^r(P-B_i)}\Theta ]\\
  &=&[t^*_{\Sigma_{i=1}^r(B_i-A_i)}\Theta-\Theta]\\
  &=&[t^*_{D_2}\Theta-\Theta].
 \end{eqnarray*}
 The last expression equals to $\lambda(a_2)$ since $D_2$ represents $a_2$.
\end{proof}

Let $f$ be a rational function on $X$ such that $\divv(f)=mD_2$. Let $\tilde{f}$ the rational function on $X^{(g)}$ defined by $P_1+\ldots+P_g\mapsto{f(P_1)\cdots f(P_g)}$. Its divisor equals to $(\sum_{i=1}^rS_{B_i}-S_{A_i})$ where $S_Q:=\{D\in X^{(g)}\mid D-Q\geq0\}$. The image of $S_Q$ under $\pi$ is the divisor $t^*_{P-Q}\Theta=\{[D]\in J\mid D+gP-Q\geq0\}$. Via $\pi$ we can consider $\tilde{f}$ as rational function on $J$ and its divisor is then precisely $-m\bar{D}_2$. We further have $\tilde{f}(a_1)=f(Q_1)\cdots f(Q_g)$ and $\tilde{f}(0)=f(P)^g$. We therefore obtain
\begin{equation}\label{eq:bmjac}
 b_m(a_1,\lambda(a_2))=\frac{f(P)^g}{f(Q_1)\cdots f(Q_g)}\in k^{\times}/ (k^\times)^m.
\end{equation}

We spend the rest of the section making this description independent of the choice of $P$. For this we need to set up some notation.

\begin{Def}
 Let $g$ a nonzero rational function on $X$ and $D=\sum n_P\cdot P$ a divisor such that $\divv(g)$ and $D$ have disjoint supports. Then we define \[g(D)=\prod g(P)^{n_P} \in k^* .\]
\end{Def}

Note that if $\deg(D)=0$, then $g(D)$ only depends on $\divv(g)$ (rather than on $g$) because multiplying $g$ with a nonzero scalar does not affect $g(D)$ in this case. The central ingredient for what follows is Weil's reciprocity law, see for example \cite[Ex.~2.11]{silverman}.

\begin{thm}[Weil reciprocity]
 Let $g_1,g_2$ nonzero rational functions on $X$ such that $\divv(g_1)$ and $\divv(g_2)$ have disjoint supports. Then $g_1(\divv(g_2))=g_1(\divv(g_2))$.
\end{thm}

\begin{thm}\label{thm:curvepairing}
 Let $a_1,a_2\in J_m(k)$ represented by divisors $D_1,D_2$ on $X$ with disjoint supports and $\lambda: J\to J^\vee$ the principal polarization defined in \Cref{eq:princpolthet} by the choice of $\Theta$. Let $f$ be a rational function on $X$ such that $\divv(f)=mD_2$. Then
 \begin{equation*}
  b_m(a_1,\lambda(a_2))=f(-D_1)\in k^{\times}/ (k^\times)^m.
 \end{equation*}
\end{thm}

\begin{proof}
 By \Cref{eq:bmjac} it suffices to show that the class of $f(D_1)$ in $k^{\times}/ (k^\times)^m$ does not depend on $D_1$ but only on its class modulo linear equivalence. Thus let $g$ a nonzero rational function on $X$ and consider
 \begin{eqnarray*}
  f(D_1+\divv(g))&=& f(D_1)\cdot f(\divv(g))\\
  &=& f(D_1)\cdot g(\divv(f))\\
  &=& f(D_1)\cdot g(D_2)^m
 \end{eqnarray*}
 where the second equation follows from Weil reciprocity.
\end{proof}

\begin{ex}
 Consider the hyperelliptic curve $X$ defined over $\Q$ by
 \begin{equation*}
  y^2=\prod_{i=0}^5(x-i).
 \end{equation*}
 Let $J$ its Jacobian and $\lambda\colon J\to J^\vee$ the principal polarization defined by a theta divisor.
 For $i=0,\ldots,5$ let $P_i=(i,0)$. The 15 nontrivial $2$-torsion points of $J$ are $a_{ij}=P_j-P_i$ for $0\leq i<j\leq 5$. 
 On these points the Weil pairing can be easily described as 
 \begin{equation*}
  e_2(a_{ij},\lambda(a_{kl}))=(-1)^{|\{i,j\}\cap\{k,l\}|}.
 \end{equation*}
Using \Cref{thm:curvepairing} one further computes that $q_2(a_{01})=5$, $q_2(a_{02})=-10$, $q_2(a_{03})=10$, $q_2(a_{04})=-5$ and $q_2(a_{05})=1$. Finally, for $0<i<j\leq5$ we have
\begin{equation*}
 q_2(a_{ij})=q_2(a_{0i}+a_{0j})=-q_2(a_{0i})q_2(a_{0j})
\end{equation*}
by \Cref{thm:q2prop}.
\end{ex}

\section{Real abelian varieties}
\subsection{Preliminaries}
We now turn to abelian varieties over $\R$. We first recall some basic properties of real abelian varieties from \cite{grossharris}.
Let $A$ be an abelian variety of dimension $g$ over $\R$. By abuse of notation, we denote the abelian group $A(\C)$ also by $A$. We denote by $A^0(\R)$ the connected component of $A(\R)$ which contains the neutral element. We further denote $A_2^0(\R)=A_2(\R)\cap A^0(\R)$.
The group $A^0(\R)$ is isomorphic to $(\R/\Z)^g$. We let $\lambda:A\to A^\vee$ a principal polarization. 
The natural map
\begin{equation}\label{eq:twoways}
 \mu_2(\C)=\{\pm1\}=\mu_2(\R)\to \R^{\times}/ (\R^\times)^{2}
\end{equation}
is an isomorphism and we identify both sides with the \emph{additive} group $\Z_2=\Z/2\Z$. The Weil pairing defines a symplectic bilinear form
\begin{equation*}
 \langle a_1,a_2\rangle:=e_2(a_1,\lambda(a_2))\in\Z_2
\end{equation*} on $A_2$
 for which $A_2^0(\R)$ is a maximal isotropic subspace. Thus we can extend a basis $v_1,\ldots,v_g$ of $A^0_2(\R)$  by some $w_1,\ldots,w_g$ to a symplectic basis of $A_2$. This choice of basis identifies $A_2$ with $\Z_2^{2g}$ in such a way that writing $a=\binom{a_u}{a_l}$ with $a_u,a_l\in\Z_2^g$ for $a\in\Z_2^{2g}$, we have $\langle a,c\rangle=a_u^tc_l-a_l^tc_u$ for $a,c\in\Z_2^{2g}$. 
\subsection{A signed count of $2$-torsion points} 
With this preparation, we are ready to prove a signed count of the real $2$-torsion points on a principally polarized real abelian variety.
 \begin{lem}\label{lem:lagr}
  The bilinear form 
  \begin{equation*}
    A_2(\R)\times A_2(\R)\to\Z_2,\, (a_1,a_2)\mapsto b_2(a_1,\lambda(a_2))
  \end{equation*}
   is trivial on $A^0_2(\R)\times A_2(\R)$.
 \end{lem}

\begin{proof}
 Let $\lambda(a_2)$ be represented by a divisor $D$ on $A$ and $f$ a rational function on $A$ such that $\divv(f)=2D$ and $f(0)=1$. The rational function $f$ has constant sign on $A^0(\R)$ since $\divv(f)=2D$ and this sign is positive since $f(0)=1$.
\end{proof}

Recall the definition $q_2(a)=b_2(a,\lambda(a))$ for $a\in A_2(\R)$. We regard $q_2(a)$ as an element of $\{\pm1\}$ via \Cref{eq:twoways}. The sum over all $q_2(a)$ only depends on $g$:

\begin{thm}\label{thm:signed2tor}
 For every principally polarized abelian variety $A$ of dimension $g$ over $\R$ we have 
 \begin{equation*}
  \sum_{a\in A_2(\R)} q_2(a)=2^g.
 \end{equation*}
\end{thm}

\begin{proof}
 By \Cref{lem:lagr} the sum of all $q_2(a)$ with $a\in A^0_2(\R)$ equals to $2^g$.
 For any $0\neq c\in \Z_2^g$ denote $$V_c=\{a\in A_2\mid a_l=c\}.$$ The set $A_2(\R)\setminus A^0_2(\R)$ is contained in the disjoint union of all $V_c$ for $0\neq c\in \Z_2^g$. Fix some $0\neq c\in \Z_2^g$ and let $b\in\Z_2^g$ a unit vector such that $c^tb=1$. Let $0 \neq d\in A^0_2(\R)$ with $d_u=b$ (and $d_l=0$). 
 The $2$-element subgroup generated by $d$ acts on $V_c$. Each orbit $\{a+md\mid m\in\Z_2\}$, $a\in V_c$, has length $2$. Since $d\in A(\R)$, either the entire orbit is contained in $A(\R)$ or it is disjoint from $A(\R)$. In the former case we have
\begin{equation*} 
 q_2(a+ d)=q_2(a)\cdot q_2(d)\cdot (-1)^{\langle a,d\rangle}=-q_2(a)
\end{equation*}
where the first equality follows from \Cref{thm:q2prop} and the second by
\begin{equation*}
 \langle a,d\rangle=a_u^td_l-a_l^td_u=c^tb=1\in\Z_2.
\end{equation*}
Thus each orbit consists of one element $a'$ with $q_2(a')=1$ and one element $a''$ with $q_2(a'')=-1$. This proves the claim.
\end{proof}

\begin{rem}
 \Cref{thm:signed2tor} remains valid when replacing $\R$ by an arbitrary real closed field $R$. The analogues over $R$ for the properties of real abelian varieties that we used in the proof can be found in the appendix of \cite{clausabel}.
\end{rem}

\begin{ex}\label{ex:el1}
 Consider the elliptic curve $E$ defined by
 \begin{equation*}
  y^2 =\frac{1}{3} (x + 3)(x^2 + 1).
 \end{equation*}
 It is depicted in \Cref{fig:elliptic} on the left.
 Besides the neutral element $O$ at infinity, this curve has one more real $2$-torsion point $Q=(-3,0)$. We consider the standard principal polarization that sends a point $P$ to the divisor class of $P-O$. We clearly have $q_2(O)=+1$. Letting $f=\frac{x^2+1}{x^2}$ we have that $\frac{1}{2}\divv(f)$ is linearly equivalent to $Q-O$. Thus we can compute $q_2(Q)$ as the sign of
 \begin{equation*}
  \frac{f(Q)}{f(O)}=\frac{10}{9}.
 \end{equation*}
 This is consistent with \Cref{thm:signed2tor} as
 \begin{equation*}
  \sum_{a\in E(\R)_2} q_2(a)=q_2(O)+q_2(Q)=1+1=2=2^g.
 \end{equation*}
\end{ex}

\begin{ex}\label{ex:el2}
 Consider the elliptic curve $E$ defined by
 \begin{equation*}
  y^2 = \frac{1}{3} x(x - 1)(x + 3).
 \end{equation*}
 It is depicted in \Cref{fig:elliptic} on the right.
 Besides the neutral element $O$ at infinity, this curve has three more real $2$-torsion points $Q_1=(-3,0)$, $Q_2=(0,0)$ and $Q_3=(1,0)$. We consider the standard principal polarization that sends a point $P$ to the divisor class of $P-O$. We clearly have $q_2(O)=+1$. Letting $f=\frac{x-1}{x}$ we have that $\frac{1}{2}\divv(f)$ is linearly equivalent to $Q_1-O$. Thus we can compute $q_2(Q_1)$ as the sign of
 \begin{equation*}
  \frac{f(Q_1)}{f(O)}=\frac{4}{3}.
 \end{equation*}
 Similarly, one computes that $q_2(Q_2)=-1$ and $q_2(Q_3)=+1$.
 This is consistent with \Cref{thm:signed2tor} as
 \begin{equation*}
  \sum_{a\in E(\R)_2} q_2(a)=q_2(O)+q_2(Q_1)+q_2(Q_2)+q_2(Q_3)=1+1-1+1=2=2^g.
 \end{equation*}
\end{ex}

\begin{figure}[ht]
 \includegraphics[width=5cm]{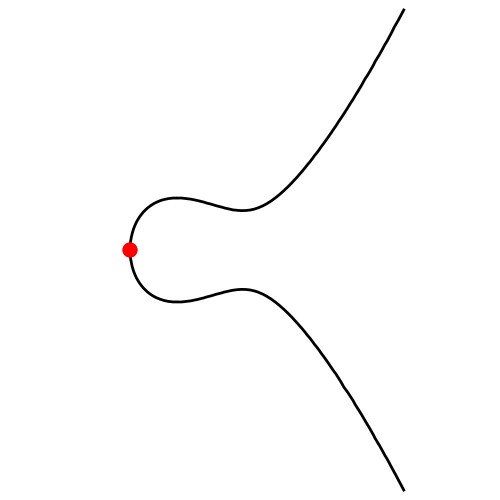} \quad
 \includegraphics[width=5cm]{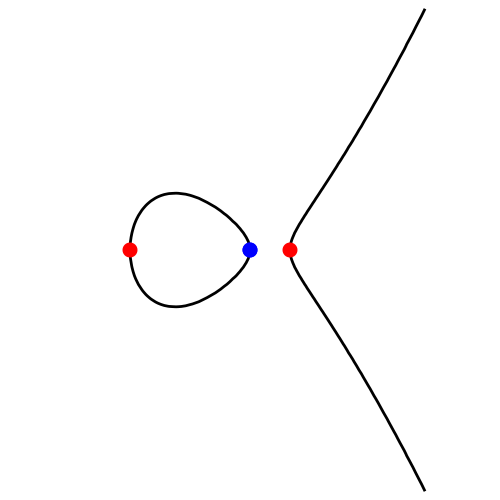}
\caption{The two elliptic curves from \Cref{ex:el1} and \Cref{ex:el2} and their non-trivial torsion points. $2$-torsion points $a$ with $q_2(a)=+1$ are marked red while $2$-torsion points $a$ with $q_2(a)=-1$ are marked blue.}
\label{fig:elliptic}
\end{figure}

\begin{ex}\label{ex:kum}
 Let $J$ be the Jacobian variety of a smooth real curve of genus two together with the principal polarization defined by the theta divisor.
 Its \emph{Kummer surface} is the quotient of $J$ by the involution $a\mapsto -a$.
 It can be embedded as a nodal surface $X$ of degree four in $\pp^3$. Its 16 singularities correspond to the $2$-torsion points of $J$ as these are the fixed points of the involution. Given a real nontrivial $2$-torsion point $a$ of $J$ we can read off $q_2(a)$ from the geometry of the Kummer surface $X$. Namely, there are exactly two real hyperplanes $H_1$ and $H_2$ that each contain six nodes of $X$, two of which are the ones corresponding to $0$ and $a$. Then $q_2(a)$ is $+1$ if and only if each connected component of $\R\pp^3\setminus(H_1\cup H_2)$ contains an even number of real singularities of $X$. \Cref{fig:kummer} displays a Kummer surface all of whose 16 nodes are real. Ten nodes correspond to a positive $2$-torsion point and the remaining six to a negative one. Since $10-6=4=2^g$, this is consistent with \Cref{thm:signed2tor}.
\end{ex}

\begin{figure}[ht]
 \includegraphics[width=10cm]{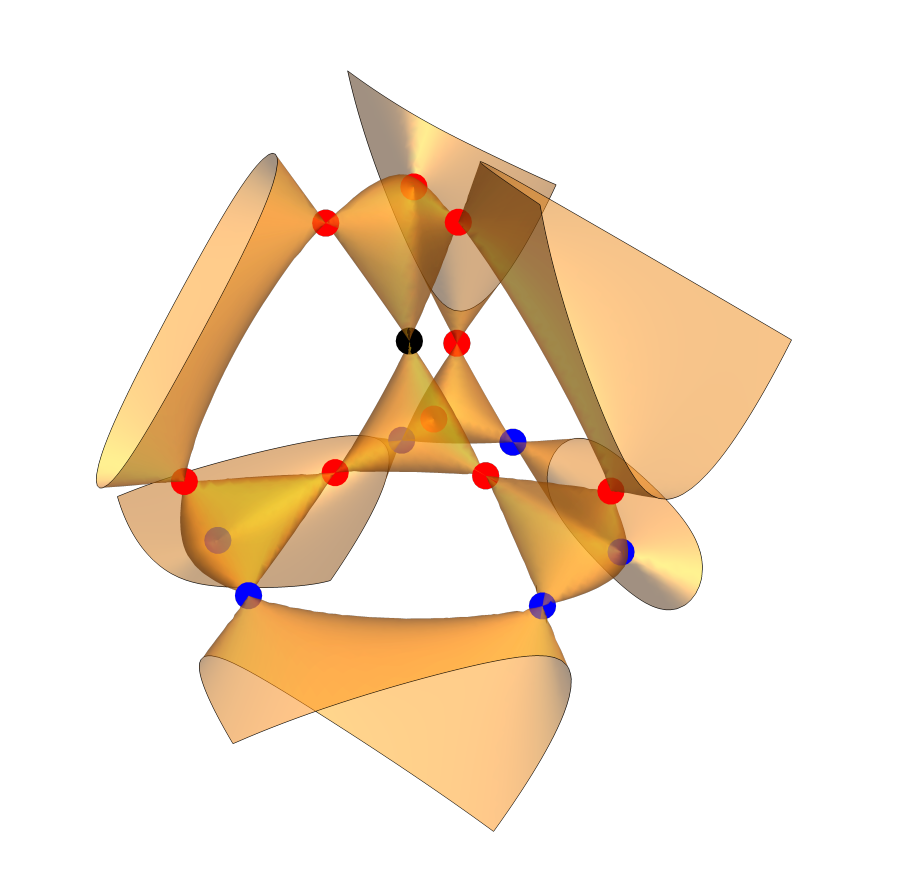}
\caption{The Kummer surface of a real curve of genus two whose real part has three connected components. All of its 16 singularities are real. Nodes that correspond to a nontrivial $2$-torsion point $a$ with $q_2(a)=+1$ are marked red while nodes that correspond to a nontrivial $2$-torsion point $a$ with $q_2(a)=-1$ are marked blue. The node which corresponds to the neutral element is marked black.}
\label{fig:kummer}
\end{figure}

\subsection{A geometric interpretation of $b_2$}
 We want to give a more topological interpretation of the pairing $b_2$.
The results from this subsection are not needed for the further progress of the paper.

Let $a\in A_2(\R)$ and $a'\in A_2^\vee(\R)$.
Let $L$ the line bundle on $A$ corresponding to $a'$ and fix an isomorphism $\psi\colon L\otimes L\to\cO_A$. One can associate an unramified double cover $\varphi\colon\tilde{A}\to A$ with $\psi$ as follows. The multiplication defined on $\cO_A\oplus L$ by
\begin{equation}\label{eq:cover}
 (a,s)\cdot(b,t):=(a\cdot b+\psi(s\otimes t),a\cdot t+b\cdot s)
\end{equation}
for sections $a,b$ of $\cO_A$ and sections $s,t$ of $L$ equips $\cO_A\oplus L$ with the structure of an $\cO_A$-algebra. Then $\tilde{A}$ is given by $\Spec(\cO_A\oplus L)$. The preimage of a real point $p\in A(\R)$ under $f$ either consists of two real points or a pair of non-real complex conjugate points.

\begin{prop}
 We have $b_2(a,a')=0\in\Z_2$ if and only if $\varphi^{-1}(a)$ has the same number of real points as $\varphi^{-1}(0)$.
\end{prop}

\begin{proof} 
 Let $D$ a divisor corresponding to $L$ such that $0$ and $a$ are not in the support of $D$.
 For every nonempty open subset $U$ of $A$ we can identify $L(U)$ with the set of rational functions $g$ on $A$ for which the restriction of $\divv(g)+D$ to $U$ is effective. Then there is a rational function $f$ on $A$ with $\divv(f)=2D$ such that the isomorphism $\psi$ is given by $\psi(g_1\otimes g_2)=\frac{g_1g_2}{f}$. Let $p\in A(\R)$ not in the support of $D$. Then, for a sufficiently small open affine neighbourhood $U$ of $p$, the $\cO_A(U)$-module $L(U)$ is generated by the element $1$. Thus by \Cref{eq:cover} the $\cO_A(U)$-algebra $\cO_{\tilde{A}}(\varphi^{-1}(U))$ is generated by $(0,1)$ and we have
 \begin{equation*}
  (0,1)\cdot(0,1)=(\frac{1}{f},0).
 \end{equation*}
Therefore, we have that
\begin{eqnarray*}
 \cO_{\tilde{A}}(\varphi^{-1}(U))\cong\cO_A(U)\left[{\sqrt{\frac{1}{f}}}\right]
\end{eqnarray*}
as $\cO_A(U)$-algebras. In particular, the point $p$ has a real preimage under $\varphi$ if and only if $f(p)>0$. Thus $\varphi^{-1}(a)$ has the same number of real points as $\varphi^{-1}(0)$ if and only if $\frac{f(p)}{f(0)}>0$ which implies the claim by the definition of $b_2$.
\end{proof}
We will see another interpretation of $b_2$ in \Cref{sec:orienta1}.

\section{Jacobians of real curves}\label{sec:realjac}
Now let $X$ be a smooth projective geometrically irreducible curve of genus $g$ over $\R$ with $X(\R)\neq\emptyset$ and let $J$ be its Jacobian which comes with a principal polarization $\lambda:J\to J^\vee$ given by the theta divisor. Let $X_0,\ldots,X_s$ the  connected components of $X(\R)$. Then $J_2(\R)\cong\Z_2^{g+s}$ by \cite[Proposition~3.2]{grossharris}. Further let $\Gal(\C/\R)=\{1,\sigma\}$. Finally, we denote $a(X)=1$ if $X(\C)\setminus X(\R)$ is connected and $a(X)=0$ otherwise.

Let $c\in J_2(\R)$ represented by a divisor $D$ on $X$. Then there is a rational function $f$ on $X$ such that $\divv(f)=2D$ which is nonnegative on $X_0$. Such $f$ has constant sign on every $X_i$ and we let $\sg_i(c)\in\Z_2$ such that this sign equals $(-1)^{\sg_i(c)}$.
This defines a group homomorphism 
\begin{equation*}
 \sg\colon J_2(\R)\to \Z_2^s,\, c\mapsto(\sg_1(c),\cdots,\sg_s(c)).
\end{equation*}
Note that $\sg_0(c)=0$ for every $c\in J_2(\R)$.
Now let $c\in\Pic(X)$, represented by a divisor $D=\sum_{P\in X}n_p\cdot P$ on $X$. It follows from \cite[Lemma~4.1]{grossharris} that for each $i$ the residue class
\begin{equation*}
 \Par_i(c):=\sum_{P\in X_i}\overline{n_P}\in\Z_2
\end{equation*}
does not depend on the representing divisor $D$.
We write
\begin{equation*}
 \Par(c):=(\Par_1(c),\ldots,\Par_s(c))\in\Z_2^s.
\end{equation*}
Note that $\Par_0(c)$ is determined by $\Par(c)$ and the degree of $c$. For instance if $c\in J(\R)=\Pic^0(X)$, then we have $\Par_0(c)=\sum_{i=1}^s \Par_i(c)$. The kernel of the group homomorphism $\Par\colon J_2(\R)\to\Z_2^s$ is $J^0_2(\R)$ \cite[proof of Proposition~3.2]{grossharris}.
Combining these two maps we obtain the homomorphism
\begin{equation*}
(\sg,\Par): J_2(\R)\to \Z_2^s\oplus\Z_2^s.
\end{equation*}

\begin{lem}\label{lem:homsurj}
 The homomorphism $(\sg,\Par)$ is surjective. Its kernel has dimension $g-s$ as $\Z_2$-vector space.
\end{lem}

\begin{proof}
 Surjectivity was shown in \cite[\S5]{geyermartens}. The statement on the kernel then follows from $J_2(\R)\cong\Z_2^{g+s}$.
\end{proof}

Our pairing $b_2$ can also be expressed in terms of $\Par$ and $\sg$.

\begin{prop}\label{prop:pairingint}
 For every $a_1,a_2\in J_2(\R)$ we have
 \begin{equation*}
  b_2(a_1,\lambda(a_2))=\sum_{i=1}^s\Par_i(a_1)\cdot\sg_i(a_2)\in\Z_2.
 \end{equation*}
\end{prop}

\begin{proof}
 Let $a_1,a_2$ represented by divisors $D_1,D_2$ on $X$ with disjoint supports. Let $f$ be a rational function on $X$ such that $\divv(f)=2D_2$. Then by \Cref{thm:curvepairing} we can compute $b_2(a_1,\lambda(a_2))$ as the square class of $f_2(-D_1)$. Let $D_1=\sum_{P\in X}n_p\cdot P$. Then $f_2(-D_1)$ has the same square class as
 \begin{equation*}
  \prod_{P\in X(\R)} f_2(P)^{n_p}=\prod_{i=0}^s\prod_{P\in X_i}f_2(P)^{n_p}.
 \end{equation*}
 The square class of $\prod_{P\in X_i}f_2(P)^{n_p}$ is nontrivial if and only if $\sum_{P\in X_i}n_p$ is odd and $f_2$ has negative sign on $X_i$. Since this is equivalent to $\Par_i(a_1)\cdot\sg_i(a_2)=1$, the claim follows.
\end{proof}

The Weil pairing defines a symplectic bilinear form
\begin{equation*}
 J_2\times J_2\to\Z_2,\, \langle a_1,a_2\rangle:=e_2(a_1,\lambda(a_2)).
\end{equation*}
In the following we will construct a specific symplectic basis of $J_2$. By \Cref{lem:homsurj} we can choose some elements $v_1,\ldots,v_s,w_1,\ldots,w_s\in J_2(\R)$ that are mapped by $(\sg,\Par)$ to the standard basis of $\Z_2^s\oplus\Z_2^s$  and a basis $v_{s+1},\ldots,v_{g}$ of $\ker(\sg,\Par)$.

\begin{lem}
 The basis $v_1,\ldots,v_g,w_1,\ldots,w_s$ of $J_2(\R)$ satisfies $\langle v_i,w_j\rangle=\delta_{ij}$, $\langle v_i,v_j\rangle=0$ and $\langle w_i,w_j\rangle=0$ for all $i=1,\ldots,g$ and $j=1,\ldots,s$.
\end{lem}

\begin{proof}
 By \Cref{prop:pairingint} we have $b_2(v_i,\lambda(w_j))=\delta_{ij}$ and $b_2(w_j,\lambda(v_i))=0$ for all $i,j$. \Cref{thm:pairing}(iii) thus implies that $\langle v_i,w_j\rangle=\delta_{ij}$. Similarly $b_2(v_i,\lambda(v_j))=0$ and $b_2(w_i,\lambda(w_j))=0$  implies $\langle v_i,v_j\rangle=0$ and $\langle w_i,w_j\rangle=0$ for all $i,j$.
\end{proof}

\begin{thm}\label{thm:basisj2}
 There is a symplectic basis $v_1,\ldots,v_g,w_1,\ldots,w_g$ of $J_2$ with the following properties.
 \begin{enumerate}[(i)]
  \item $v_1,\ldots,v_g,w_1,\ldots,w_s$ is a basis of $J_2(\R)$.
  \item $\sg_j(v_i)=\delta_{ij}$ and $\Par_j(v_i)=0$ for all $i=1,\ldots,g$ and $j=1,\ldots,s$.
  \item $\sg_j(w_i)=0$ and $\Par_j(w_i)=\delta_{ij}$ for all $i,j=1,\ldots,s$.
  \item The representing matrix of $\sigma\colon J_2\to J_2$ with respect to this basis is of the form
  \begin{equation*}
   \begin{pmatrix}
    I_g& H\\
    0 & I_g
   \end{pmatrix}
  \end{equation*}
  where $I_g$ is the $g\times g$ identity matrix and $H$ is a block matrix 
  \begin{equation*}
   H=\begin{pmatrix}
    0& 0\\
    0 & H'
   \end{pmatrix}
  \end{equation*}
  of rank $g-s$. Here $H'$ is the $(g-s)\times(g-s)$ identity matrix if $a(X)=1$ and $H'$ is a block diagonal matrix with blocks $\begin{pmatrix}0&1\\1&0\end{pmatrix}$ if $a(X)=0$.
 \end{enumerate}
\end{thm}

\begin{proof}
By the preceding lemma we can complete $v_1,\ldots,v_g,w_1,\ldots,w_s$ to a symplectic basis $v_1,\ldots,v_g$, $w_1,\ldots,w_g$ of $J_2$. By construction this basis satisfies $(i)$---$(iii)$. Because $\sigma\colon J_2\to J_2$ preserves the symplectic form and fixes $v_1,\ldots,v_g,w_1,\ldots,w_s$, its representing matrix with respect to our chosen basis is of the form
  \begin{equation*}
   \begin{pmatrix}
    I_g& H\\
    0 & I_g
   \end{pmatrix}
  \end{equation*}
  where $I_g$ is the $g\times g$ identity matrix and $H$ is a block matrix 
  \begin{equation*}
   H=\begin{pmatrix}
    0& 0\\
    0 & H'
   \end{pmatrix}
  \end{equation*}
  of rank $g-s$ and $H'$ is symmetric. As in \cite[Proposition~4.4]{grossharris} we can bring $H'$ into the desired form by a change of coordinates on $w_{s+1},\ldots,w_g$ and the corresponding change of coordinates on $v_{s+1},\ldots,v_g$,
\end{proof}

\section{Real theta characteristics}\label{sec:theta}
Let $X$ be a smooth projective geometrically irreducible curve of genus $g$ over $\R$ with $X(\R)\neq\emptyset$ and let $J$ be its Jacobian. Again let $\Gal(\C/\R)=\{1,\sigma\}$. Further let $X_0,\ldots,X_s$ the  connected components of $X(\R)$. We fix a symplectic basis $v_1,\ldots,v_g,w_1,\ldots,w_g$ of $J_2$ as in \Cref{thm:basisj2}. Every $c\in J_2$ can be written as
\begin{equation*}
 c=\sum_{i=1}^g a_i v_i + \sum_{i=1}^g b_i w_i
\end{equation*}
and we define $c_u:=(a_1,\ldots,a_g)^t\in\Z_2^g$ and $c_l:=(b_1,\ldots,b_g)^t\in\Z_2^g$.

Recall that a \emph{quadratic form with polarity $\langle-,-\rangle$} (here $\langle a,b\rangle=e_2(a,\lambda(b))$ as in the previous section) on $J_2$ is a map $q\colon J_2\to\Z_2$ such that
\begin{equation*}
 q(a_1+a_2)=q(a_1)+q(a_2)+\langle a_1,a_2\rangle
\end{equation*}
for all $a_1,a_2\in J_2$. We denote by $Q(J_2)$ the set of all quadratic forms with polarity $\langle-,-\rangle$ on $J_2$.

\begin{Def}
 By a \emph{theta characteristic} on $X$ we mean a line bundle $L$ on $X_\C$ such that $L\otimes L$ is isomorphic to the canonical line bundle on $X$. A theta characteristic is called \emph{even} or \emph{odd} according to the parity of the dimension of its space of global sections. We say that $L$ is a \emph{real theta characteristic} if it arises as base change from a line bundle on $X$. This is equivalent to being fixed under $\sigma$.
\end{Def}
As a reference for the following statements on theta characteristics we recommend \cite[Chapter~5]{dolga}.
For every theta-characteristic $\nu\in\Pic(X_\C)$ the map
\begin{equation*}
 q_\nu\colon J_2\to\Z_2,\, a\mapsto h^0(a+\nu)+h^0(\nu)
\end{equation*}
is a quadratic form with polarity $\langle-,-\rangle$. Moreover, the map $\nu\mapsto q_\nu$ is a bijection. If $\nu$ is a theta characteristic and $a\in J_2$, then $a+\nu$ is a again a theta characteristic and we have for all $c\in J_2$:
\begin{equation}\label{eq:polar1}
 q_{a+\nu}(c)=q_\nu(c)+\langle a,c\rangle.
\end{equation}
A theta characteristic $\nu$ is even resp. odd if and only if $\arf(q_\nu)=0$ resp. $\arf(q_\nu)=1$. It is real if and only if $q_\nu=q_\nu\circ\sigma$.

A specific element of $Q(J_2)$ is the following
\begin{equation*}
 q_0\colon J_2\to\Z_2,\, c\mapsto c_u^t\cdot c_l.
\end{equation*}
We denote the theta characteristic corresponding to $q_0$ by $\eta_0$.
Note that by \Cref{prop:pairingint} and \Cref{thm:basisj2} we have $q_0(c)=q_2(c)$ for all $c\in J_2(\R)$. For every $a\in J_2$ we denote $q_a:=q_{a+\eta_0}$. Then \Cref{eq:polar1} translates to
\begin{equation}\label{eq:polar2}
 q_{a}(c)=q_0(c)+\langle a,c\rangle.
\end{equation}
The map $J_2\to Q(J_2),\, a\mapsto q_a$ is a bijection and one has $\arf(q_a)=q_0(a)$. In particular, the  theta characteristic $\eta_0$ is even.

\begin{lem}\label{lem:thetainv}
 We have
 \begin{equation*}
  q_0\circ\sigma=q_h
 \end{equation*}
 where $h=\sum_{i=s+1}^g v_i$ if $a(X)=1$ and $h=0$ if $a(X)=0$. 
\end{lem}

\begin{proof}
 For $c\in J_2$ we have that $\sigma(c)$ is given in coordinates
 \begin{equation*}
  \begin{pmatrix}
    I_g& H\\
    0 & I_g
   \end{pmatrix}
   \begin{pmatrix}
    c_u\\
    c_l 
   \end{pmatrix}=\begin{pmatrix}
    c_u+Hc_l\\
    c_l
   \end{pmatrix}
 \end{equation*}
 where $H$ is the matrix from \Cref{thm:basisj2}(iv).
 Therefore, we have 
 \begin{equation*}
  q_0(\sigma(c))=q_0(c)+c_l^tHc_l.
 \end{equation*}
  By our definition of $h$, we have $c_l^tHc_l=\langle h,c\rangle$. This implies by \Cref{eq:polar2}
  \begin{equation*}
   q_0(\sigma(c))=q_0(c)+\langle h,c\rangle=q_h(c).\qedhere
  \end{equation*}
\end{proof}

\begin{cor}\label{cor:conju}
 For every $a\in J_2$ we have $q_a\circ\sigma=q_{\sigma(a)+h}$ where $h$ is defined as in \Cref{lem:thetainv}.
\end{cor}

\begin{proof}
 For $c\in J_2$ we have
 \begin{equation*}
  q_a(\sigma(c))=q_0(\sigma(c))+\langle a,\sigma(c)\rangle=q_h(c)+\langle \sigma(a),c\rangle
 \end{equation*}
Here the first equality is \Cref{eq:polar2} and the second is \Cref{lem:thetainv} together with the fact that complex conjugation is self-adjoint with respect to the Weil pairing. Then \Cref{eq:polar1} implies the claim.
\end{proof}

\begin{cor}\label{cor:whenthetareal}
 For all $c\in J_2$, the theta characteristic $c+\eta_0$ is real if and only if the last $g-s$ entries of $c_l$ are equal to $a(X)$. In particular, the theta characteristic $\eta_0$ is real if and only if $a(X)=0$. 
\end{cor}

\begin{proof}
 By \Cref{cor:conju} the theta characteristic $c+\eta_0$ is real if and only if 
 \begin{equation*}
  c=\sigma(c)+h
 \end{equation*}
 where $h$ is defined as in \Cref{lem:thetainv}. Expressed in the coordinates of our chosen symplectic basis $v_1,\ldots,v_g,w_1,\ldots,w_g$ of $J_2$, this means
 \begin{equation*}
  \begin{pmatrix}
   c_u\\c_l
  \end{pmatrix}
= \begin{pmatrix}
    I_g& H\\
    0 & I_g
   \end{pmatrix} \begin{pmatrix}
   c_u\\c_l
  \end{pmatrix}+\begin{pmatrix}
   h_u\\h_l
  \end{pmatrix}.
 \end{equation*}
 Since $h_l$ is zero, this reduces to
 \begin{equation*}
  Hc_l+h_u=0.
 \end{equation*}
 Now the claim follows from the description of $H$ in \Cref{thm:basisj2}.
\end{proof}

\begin{Def}[\cite{knebuschcurves}]
 A rational differential $\omega\neq0$ on $X$ is called \emph{definite} if $\omega$ has even order at all real points. It is called \emph{strictly definite} if it has no real poles or zeros.  Two definite differentials $\omega$ and $\omega'$ are \emph{equivalent} if $\omega'=\pm f\cdot\omega$ for some real rational function $f$ that is nonnegative on $X(\R)$ (wherever it is defined). A \emph{semi-orientation} on $X$ is an equivalence class $[\omega]$ of definite differentials on $X$.
\end{Def}

\begin{rem}
 There are exactly $2^s$ different semi-orientations on $X$ and each one contains a strictly definite differential by \cite[p.63]{knebuschcurves}. A strictly definite differential defines a volume form on $X(\R)$ and thus defines an orientation in the classical sense. Therefore, a semi-orientation is an equivalence class of orientations modulo global reversion and vice versa.  The \emph{complex semi-orientation} is the equivalence class of the orientation on $X(\R)$ induced by an orientation of one of the connected components of $X(\C)\setminus X(\R)$ modulo global reversion.
\end{rem}

\begin{Def}
 If $\nu$ is a real theta characteristic of $X$ represented by a divisor $D$, then $2D$ is the divisor of a definite differential $\omega$. The semi-orientation $[\omega]$ does not depend on the representing divisor $D$. We say that \emph{$[\omega]$ is induced by $\nu$} and we write $[\nu]:=[\omega]$. 
\end{Def}

\begin{rem}\label{rem:holonotcomplex}
 It was pointed out in \cite[Proof of Proposition~4.2]{vin93} that if $\omega$ is a regular definite differential, then $[\omega]$ is not the complex semi-orientation. In particular, the complex semi-orientation is not induced by any real odd theta characteristic.
\end{rem}

\begin{thm}\label{thm:thetaa1}
 Let $a(X)=1$. Let $\Omega$ be a semi-orientation on $X$ and $\epsilon\in\Z_2^s$. There are exactly $2^{g-s-1}$ even and $2^{g-s-1}$ odd real theta characteristics $\nu$ on $X$ such that $[\nu]=\Omega$ and $\Par(\nu)=\epsilon$.
\end{thm}

\begin{proof}
 By \Cref{cor:whenthetareal} the set of real theta characteristics corresponds to the set of all $q_c$ such that the last $g-s$ entries of $c_l$ are equal to $1$. Since we chose our basis according to \Cref{thm:basisj2}, there are $u_1,u_2\in\Z_2^s$ such that the set of real theta characteristics $\nu$ on $X$ such that $[\nu]=\Omega$ and $\Par(\nu)=\epsilon$ corresponds to the set of all $q_c$ where
 \begin{enumerate}[(i)]
  \item the first $s$ entries of $c_u$ are equal to $u_1$,
  \item the first $s$ entries of $c_l$ are equal to $u_2$, and
  \item the last $g-s$ entries of $c_l$ are each equal to $1$.
 \end{enumerate}
 Thus there are in total $2^{g-s}$ such real theta characteristics. Finally, such a theta characteristic is even if and only if the sum of the last $g-s$ entries of $c_u$ is equal to $u_1^tu_2$. Thus there are $2^{g-s-1}$ such real even theta characteristics. The remaining ones are odd which implies the claim.
\end{proof}

\begin{Def}
 Let $\Omega_1, \Omega_2$ two semi-orientations on $X$. Let $\omega_1,\omega_2$ rational differentials that induce $\Omega_1, \Omega_2$ such that $\omega_2=f\cdot\omega_1$ for a rational function $f$ which is nonnegative on $X_0$. For $j=0,\ldots,s$ we say that $\Omega_1$ and $\Omega_2$ \emph{agree} (\emph{differ}) on $X_j$ if $f$ is nonnegative (nonpositive) on $X_j$. In particular, by definition, any two semi-orientations agree on $X_0$.
\end{Def}

\begin{thm}\label{thm:thetaa0}
 Let $a(X)=0$. Let $\Omega$ be a semi-orientation on $X$ and $\epsilon\in\Z_2^s$. 
 Let $n$ the number of indices $i\in\{1,\ldots,s\}$ such that both $\epsilon_i=0$ and $\Omega$ differs from the complex semi-orientation on $X_i$.
 There are exactly $2^{g-s}$ real theta characteristics $\nu$ on $X$ such that $[\nu]=\Omega$ and $\Par(\nu)=\epsilon$. These are all odd resp. even depending on whether 
  $n$ is odd or even.
\end{thm}

\begin{proof}
 Recall that $\eta_0$ is the theta characteristic corresponding to the quadratic form $q_0$. We have already seen that $\eta_0$ is real and even. Let $\Omega_0$ the semi-orientation induced by $\eta_0$ and $\tau=\Par(\eta_0)\in\Z_2^s$.
 By \Cref{cor:whenthetareal} the set of real theta characteristics corresponds to the set of all $q_c$ such that the last $g-s$ entries of $c_l$ are equal to $0$. 
 Let $u_1\in\Z_2^g$ be the vector whose $i$ths entry is $0$ if and only if $\Omega$ agrees with $\Omega_0$ on $X_i$. Similarly, let $u_2=\epsilon-\tau\in\Z_2^g$.
 Since we chose our basis according to \Cref{thm:basisj2},  the set of real theta characteristics $\nu$ on $X$ such that $[\nu]=\Omega$ and $\Par(\nu)=\epsilon$ corresponds to the set of all $q_c$ where
 \begin{enumerate}[(i)]
  \item the first $s$ entries of $c_u$ are equal to $u_1$,
  \item the first $s$ entries of $c_l$ are equal to $u_2$, and
  \item the last $g-s$ entries of $c_l$ are each equal to $0$.
 \end{enumerate}
 Thus there are $2^{g-s}$ such real theta characteristics. These are all even or odd depending on whether $u_1^tu_2$ is zero or one. The case $u_1=0$ implies in particular that there is no real odd theta characteristic which induces $\Omega_0$. By \Cref{rem:holonotcomplex} the complex semi-orientation is not induced by a real odd theta characteristic. Thus $\Omega_0$, being the only semi-orientation not induced by a real odd theta characteristic, must be the complex semi-orientation. Similarly, the case $u_2=0$ implies that there is no odd real theta characteristic $\nu$ with $\Par(\nu)=\tau$. It was shown in \cite[p.169]{grossharris} that this implies $\tau=(1,\ldots,1)$.  
 This shows that $u_1^tu_2$ is zero resp. one if and only if $n$ is even resp. odd.
\end{proof}

The proof of \Cref{thm:thetaa0} has shown:

\begin{cor}\label{cor:distinguished}
 If $a(X)=0$, then the theta characteristic $\eta_0$ corresponding to $q_0$ induces the complex semi-orientation and $\Par(\eta_0)=(1,\ldots,1)$.
\end{cor}

\subsection{Totally real theta characteristics}
There has been recent interest in \emph{totally real} theta characteristics.

\begin{Def}
 A theta characteristic $\nu$ on $X$ is called \emph{totally real} if $\nu=[P_1+\ldots+P_{g-1}]$ for some $P_i\in X(\R)$.
\end{Def}

The results from the previous section can be used to reprove the bounds from \cite{kumtheta}.

\begin{cor}$\,$
 \begin{enumerate}[(i)]
  \item Let $s=g-2$ and $a(X)=1$. Each semi-orientation is induced by at least $2$ totally real odd theta characteristics. In total, there are at least $2^{g-1}$ totally real odd theta characteristics.
  \item Let $s=g-1$. Each semi-orientation is induced by at least $g$ totally real odd theta characteristics. In total, there are at least $g\cdot 2^{g-1}$ totally real odd theta characteristics.
  \item Let $X$ be an $M$-curve, i.e. $s=g$. Then there are at least $\frac{g\cdot (g-1)}{2}\cdot 2^{g-1}$ totally real odd theta characteristics.
 \end{enumerate}
\end{cor}

\begin{proof}
 Let $s=g-2$ and $a(X)=1$. By \Cref{thm:thetaa1} every semi-orientation  is induced by   $2$ odd real theta characteristics that have odd degree on all $g-1$ components of $X(\R)$. Such theta characteristics are totally real. This proves $(i)$.
 
 Let $s=g-1$ which implies $a(X)=1$. By \Cref{thm:thetaa1} every semi-orientation  is induced by  one odd real theta characteristic which has odd degree on some chosen $g-1$ components of $X(\R)$. Such theta characteristics are totally real and there are $g$ possibilities of choosing $g-1$ out of $g$ components. This proves $(ii)$.
 
 Let $s=g$ which implies $a(X)=0$. For each choice of $g-1$ components of $X(\R)$ there are $2^{g-1}$ real odd theta characteristics that have odd degree precisely on these components by \Cref{thm:thetaa0}.
\end{proof}

We can summarize the three cases of the preceding corollary as follows.

\begin{cor}
 If $g\leq s+a(X)+1$, then there are at least 
 \begin{equation*}
 \binom{s+1}{g-1}\cdot 2^{g-1} 
 \end{equation*}
  totally real odd theta characteristics.
\end{cor}

Trivial upper bounds on the number of totally real odd theta characteristics are simply given by the number of real odd theta characteristics. It is not known whether these upper and lower bounds are sharp.

\begin{frag}
 Are there curves of any possible type $(g,s,a)$ such that all of its real odd theta characteristics are totally real?
\end{frag}

\begin{frag}
 If $g\leq s+a(X)+1$, are there curves of type $(g,s,a)$ such that exactly $\binom{s+1}{g-1}\cdot 2^{g-1}$ of its real odd theta characteristics are totally real?
\end{frag}

\begin{frag}
 If $g>s+a(X)+1$, are there curves of type $(g,s,a)$ such that none of its real odd theta characteristics are totally real?
\end{frag}

\subsection{Signed counts of real odd theta characteristics}
In the spirit of \Cref{thm:signed2tor} we can also derive some signed counts for real odd theta characteristics which only depend on the genus $g$ of $X$.

\begin{lem}\label{lem:lagrarf}
 Let $\nu$ be a real theta characteristic on $X$. If $a(X)=0$, then we assume that $\nu$ does not induce the complex semi-orientation.
 There are exactly $2^{g-1}$ elements $a\in J_2^0(\R)$ such that $a+\nu$ is a real odd theta characteristic.
\end{lem}

\begin{proof}
 Let $c\in J_2$ such that $q_c=q_\nu$. Our assumptions imply that $c_l\in\Z_2^g$ is not the zero vector. For $a\in J_2$ the condition $a\in J_2^0(\R)$ translates to $a_l=0$ and $a+\nu$ being odd to 
 \begin{equation*}
 1=\arf(q_{a+c})=(a+c)_u^t c_l \Leftrightarrow a_u^t c_l=1+c_u^tc_l.
 \end{equation*}
 Since $c_l\neq0$, there are exactly $2^{g-1}$ such $a$.
\end{proof}

\begin{thm}\label{thm:oddcount}
 Let $\nu$ a real theta characteristic on $X$. If $a(X)=0$, then we assume that $\nu$ does not induce the complex semi-orientation.
 Then
 \begin{equation*}
  \sum_{\eta\textrm{ real odd theta characteristic}}q_2(\eta-\nu)=2^{g-1}.
 \end{equation*}
\end{thm}

\begin{proof}
 Let $c\in J_2$ such that $q_c=q_\nu$, i.e. $c=\nu-\eta_0$.  Our assumptions imply that $c_l\in\Z_2^g$ is not the zero vector. We have
 \begin{equation*}
  \{\eta-\nu\mid \eta\textrm{ real odd theta characteristic}\}=\{b\in J_2(\R)\mid \arf(q_{b+c})=1\}=:T_c.
 \end{equation*}
 Indeed, let $\eta$ a theta characteristic and $b=\eta-\nu$. 
 Since $\nu$ is real, we have that $\eta$ is real if and only if $b$ is real. Further $\eta$ is odd if and only if $\arf(q_{\eta})=1$ and we have $q_\eta=q_{b+c}$ since $\eta=b+c+\eta_0$. Thus we have to show that
 \begin{equation*}
  \sum_{b\in T_c} q_2(b)=2^{g-1}.
 \end{equation*}
 Since $q_2(b)=1$ for $b\in J_2^0(\R)$ it suffices by \Cref{lem:lagrarf} to group the elements of $T_c\setminus J_2^0(\R)$ into disjoint pairs $\{b,b'\}$ with $q_2(b)=-q_2(b')$. We proceed as in the proof of \Cref{thm:signed2tor}. The set $T_c\setminus J_2^0(\R)$ is a disjoint union of sets 
 \begin{equation*}
 T_{c,d}=\{b\in J_2(\R)\mid b_l=d,\, (c_u+b_u)^t(c_l+d)=1\} 
 \end{equation*}
 where $d\in\Z_2^g\setminus\{0\}$. For every such $d$ there exists a $0\neq f\in J_2^0(\R)$ with $f_u^t\cdot(c_l+d)=0$ and $f_u^t\cdot d=1$ because $c_l\neq0$.
 The $2$-element group $\{0,f\}$ then acts on $T_{c,d}$ by addition. Each orbit $\{b,b+f\}$ has length two and since $f\in J(\R)$, either both its elements are real or none is. In the former case we have $$q_2(b+f)=q_2(b)\cdot q_2(f)\cdot (-1)^{\langle b+f,b\rangle}=q_2(b)\cdot(-1)^{f_u^t d}=-q_2(b).$$This proves the claim.
\end{proof}

Since no real odd theta characteristic induces the complex semi-orientation, we obtain:
\begin{cor}
 Let $\nu$ a real odd theta characteristic on $X$. 
 Then
 \begin{equation*}
  \sum_{\eta\textrm{ real odd theta characteristic}}q_2(\eta-\nu)=2^{g-1}.
 \end{equation*}
\end{cor}

We give an interpretation of these signed counts in the next section.

\section{Contact hyperplanes to canonical curves}
In this section let $X\subset\pp^{g-1}$ always denote a non-hyperelliptic smooth geometrically irreducible projective real curve of genus $g$ with $X(\R)\neq\emptyset$ embedded via its canonical embedding. 
As in the previous section let $J$ be its Jacobian and $X_0,\ldots,X_s$ the  connected components of $X(\R)$.
In this situation odd theta characteristics can be interpreted as contact hyperplanes to $X$.

\begin{Def}
 A \emph{contact hyperplane} of $X$ is a hyperplane which interesects $X$ at every intersection point with even multiplicity. 
\end{Def}

In the following we assume that $X$ has only finitely many contact hyperplanes. This corresponds to the condition that $h^0(\nu)\in\{0,1\}$ for every theta characteristic $\nu$ of $X$ and is true for a general curve. In this case there is a natural bijection between (real) odd theta characteristics and (real) contact hyperplanes. Thus \Cref{thm:oddcount} gives a signed count of real contact hyperplanes which we will now interpret in terms of the extrinsic geometry of $X\subset\pp^{g-1}$. 

\begin{Def}
 A rational function $f$ on $X$ is called \emph{definite} if it has even order at all real points. 
 Such $f$ has constant sign on $X_i$ for $i=0,\ldots,s$ and we define $\sg_i(f)\in\Z_2$ such that $f$ has sign $(-1)^{\sg_i(f)}$ on $X_i$. Letting $\divv(f)=\sum_{P\in X}n_p\cdot P$ we further define $\Par_i(f):=\sum_{P\in X_i}\frac{n_P}{2}\in\Z_2$ for $i=0,\ldots,s$.
\end{Def}

\begin{rem}\label{rem:oldparnew}
 If $\divv(f)=2D$ and $c\in J_2(\R)$ is the $2$-torsion point represented by $D$, then we have $\sg_i(c)=\sg_0(f)+\sg_i(f)$ and $\Par_i(f)=\Par_i(c)$ for $i=0,\ldots,s$.
\end{rem}

\begin{Def}
 A \emph{weak contact hyperplane} of $X$ is a real hyperplane which interesects $X(\R)$ at every intersection point with even multiplicity.
\end{Def}

\begin{rem}
 Let $H_1,H_2$ two different weak contact hyperplanes of $X$. Then $\R\pp^{g-1}\setminus(H_1\cup H_2)$ has two connected components and each $X_i$ is contained in the closure of one such connected component.
\end{rem}

When we count contact points of a (weak) contact hyperplane $H$ with certain properties, we always count with multiplicities: If $H$ interesects $X$ in a point $P$ with multiplicity $2k$, we count $P$ as $k$-fold contact point of $H$ with $X$.

\begin{Def}\label{def:signofpair}
 Let $H_1,H_2$ two different weak contact hyperplanes of $X$ and let $A$ be the  connected component of $\R\pp^{g-1}\setminus(H_1\cup H_2)$ whose closure does not contain $X_0$. 
 Let $n$ be the number of real contact points to $H_1$ and $H_2$ that lie on a connected component of $X(\R)$ which is contained in the euclidean closure of $A$.
 We define the \emph{type} of the pair of weak contact hyperplanes $H_1$ and $H_2$ as 
 \begin{equation*}
  \textrm{Type}_{X_0}(H_1,H_2)=(-1)^{n}.
 \end{equation*}
 If $H_1=H_2$, then we let $\textrm{Type}_{X_0}(H_1,H_2)=1$.
\end{Def}

\begin{figure}[ht]
 \includegraphics[width=7cm]{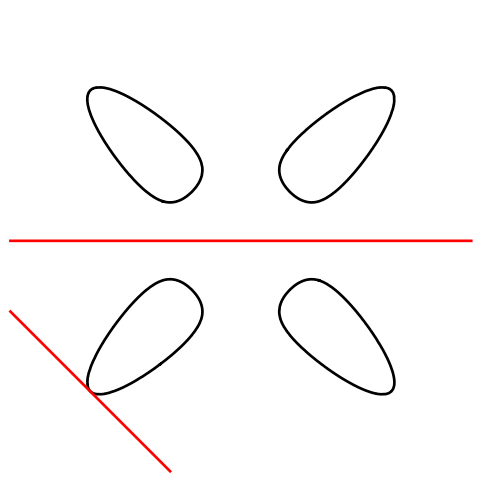}
\caption{The red lines are two weak contact hyperplanes of a genus three curve canonically embedded as a plane quartic. If $X_0$ is one of the upper two ovals, then of the type of this pair of lines is $-1$ and otherwise $+1$.}
\label{fig:weak}
\end{figure}

\begin{lem}\label{lem:tranl1}
 Let $H_1,H_2$ two different weak contact hyperplanes of $X$ cut out by linear forms $L_1$ and $L_2$. Letting $f=\frac{L_1}{L_2}$ we have $\textrm{Type}_{X_0}(H_1,H_2)=(-1)^{n}$ where
 \begin{equation*}
  n=\sg_0(f)\cdot \sum_{i=0}^s \Par_i(f)+\sum_{i=0}^s \sg_i(f)\cdot\Par_i(f)\in \Z_2.
 \end{equation*}
 If the total number of real contact points to $H_1$ and $H_2$ is even, then 
 \begin{equation*}
  n=\sum_{i=0}^s \sg_i(f)\cdot\Par_i(f)\in \Z_2.
 \end{equation*}
\end{lem}

\begin{proof}
 Let $\tilde{f}=(-1)^{\sg_0(f)}\cdot f$. Then $\sg_i(\tilde{f})=1$ if $X_i$ is contained in the euclidean closure of $A$ and $\sg_i(\tilde{f})=0$ otherwise. Thus the number of real contact points to $H_1$ and $H_2$ that lie on a connected component of $X(\R)$ which is contained in the euclidean closure of $A$ equals
 \begin{equation*}
  \sum_{i=0}^s \sg_i(\tilde{f})\cdot\Par_i(f)
 \end{equation*}
modulo $2$. Now the first claim follows since $\sg_i(f)=\sg_0(f)+\sg_i(\tilde{f})$. If the total number of real contact points to $H_1$ and $H_2$ is even, then $\sum_{i=0}^s\Par_i(f)=0$ which shows the second claim.
\end{proof}

\begin{thm}\label{thm:cont1}
 Let $H_\infty$ a weak contact hyperplane. Then
 \begin{eqnarray*}
  \sum_{H\textrm{ real contact hyperplane}}\textnormal{Type}_{X_0}(H,H_\infty)=2^{g-1}.
 \end{eqnarray*}
\end{thm}

\begin{proof}
 Let $L_\infty$ a linear form that cuts out the hyperplane $H_\infty$. Since $X$ is canonically embedded, we can think $L_\infty$ as a regular definite differential on $X$. We let $\Omega=[L_\infty]$ the corresponding semi-orientation. If $a(X)=0$, then $\Omega$ is not the complex semi-orientation  by \Cref{rem:holonotcomplex}. By \Cref{thm:thetaa1} and \Cref{thm:thetaa0} there is a (possibly even) real theta characteristic $\nu$ on $X$ such that $[\nu]=\Omega$ and $\Par_i(\nu)$ is congruent modulo $2$ to the number of contact points of $H_\infty$ on $X_i$ for $i=1,\ldots,s$. We claim that for every real contact hyperplane $H$ of $X$ we have $\textnormal{Type}_{X_0}(H,H_\infty)=q_2(\eta-\nu)$ where $\eta$ is the real odd theta characteristic corresponding to $H$. This will imply the statement by \Cref{thm:oddcount}.
 
 Let $L_0$ be a rational differential on $X$ whose divisor is $2D_0$ where $D_0$ is a divisor representing $\nu$. After replacing $L_0$ by $-L_0$ if necessary, there is a nonnegative rational function $g$ on $X$ such that $L_0=g\cdot L_\infty$ since $[L_0]=[L_\infty]$. By our choice of $\nu$, we have $\Par_i(g)=0$ for $i=1,\ldots,s$. Let $H$ be cut out by the linear form $L$. After replacing $L$ by $-L$ if necessary, we can assume that $\frac{L}{L_\infty}$ is nonnegative on $X_0$. Then by \Cref{lem:tranl1} we have $\textnormal{Type}_{X_0}(H,H_\infty)=(-1)^n$ where
 \begin{eqnarray*}
  n&=&\sum_{i=1}^s \sg_i\left(\frac{L}{L_\infty}\right)\cdot\Par_i\left(\frac{L}{L_\infty}\right)\\
  &=&\sum_{i=1}^s \sg_i\left(g\cdot\frac{L}{L_0}\right)\cdot\Par_i\left(g\cdot\frac{L}{L_0}\right)\\
  &=&\sum_{i=1}^s \sg_i\left(\frac{L}{L_0}\right)\cdot\Par_i\left(\frac{L}{L_0}\right)\\
  &=&\sum_{i=1}^s \sg_i\left(\eta-\nu\right)\cdot\Par_i\left(\eta-\nu\right)\\
  &=&q_2(\eta-\nu).
 \end{eqnarray*}
 Here the third equality holds because $g$ is nonnegative on $X(\R)$ and $\Par_i(g)=0$ for $i=1,\ldots,s$. The fourth equality follows from \Cref{rem:oldparnew} and the last one is \Cref{prop:pairingint}.
\end{proof}

\begin{rem}\label{rem:choiceA}
If in the situation of \Cref{def:signofpair} the total number of real contact points to $H_1$ and $H_2$ is even, then $\textrm{Type}_{X_0}(H_1,H_2)$ does not depend on the choice of $X_0$ and we write $\textrm{Type}(H_1,H_2):=\textrm{Type}_{X_0}(H_1,H_2)$. This is for instance the case when both $H_1$ and $H_2$ are contact hyperplanes: There are $g-1$ contact points to each of the $H_i$ and thus there are in total $2(g-1)$ contact points to $H_1$ and $H_2$. Since non-real contact points come in complex conjugate pairs, the total number of real contact points to $H_1$ and $H_2$ is even. 
\end{rem}

\begin{figure}[ht]
 \includegraphics[width=7cm]{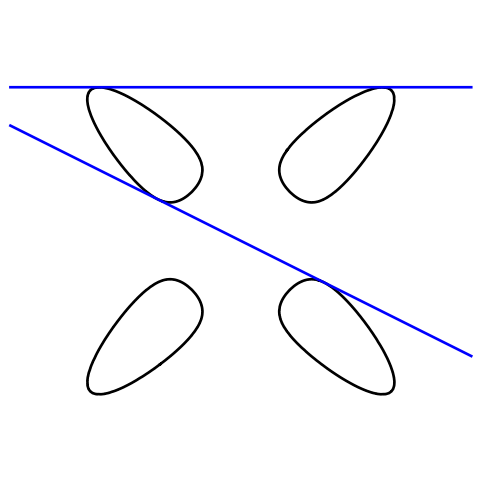}
\caption{The blue lines are bitangents of a genus three curve canonically embedded as a plane quartic. The type of this pair of lines is $-1$, independent of the choice of $X_0$.}
\label{fig:bit1}
\end{figure}

\begin{cor}\label{cor:cont2}
 Fix a real contact hyperplane $H_0$ of $X$. Then
 \begin{eqnarray*}
  \sum_{H\textrm{ real contact hyperplane}}\textnormal{Type}(H,H_0)=2^{g-1}.
 \end{eqnarray*}
\end{cor}

Another special case of interest arises from fixing a real hyperplane $H_\infty$ which does not interesect $X$ in real points. By definition $H_\infty$ is a weak contact hyperplane. This corresponds to the choice of an affine chart of real projective space in which $X(\R)$ is compact. We identify $\R^{g-1}=\R\pp^{g-1}\setminus H_\infty$. 
If $H$ is another weak contact hyperplane of $X$ and $A$ the connected component of $\R^{g-1}\setminus H$ whose closure does not contain $X_0$, then 
 \begin{equation*}
  \textrm{Type}_{X_0}(H,H_\infty)=(-1)^{n}
 \end{equation*}
 where $n$ be the number of real contact points to $H$ that lie on a connected component of $X(\R)$ which is contained in the euclidean closure of $A$.

\begin{figure}[ht]
 \includegraphics[width=5cm]{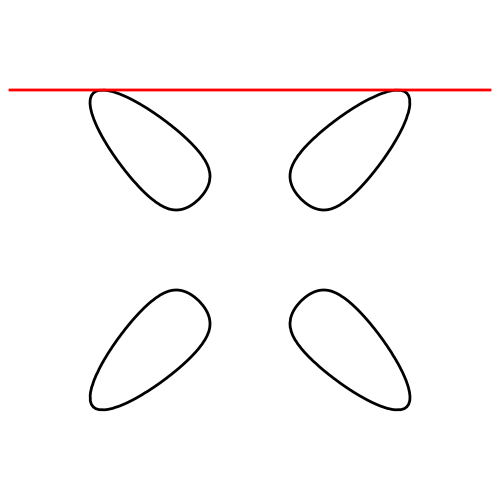} \quad
 \includegraphics[width=5cm]{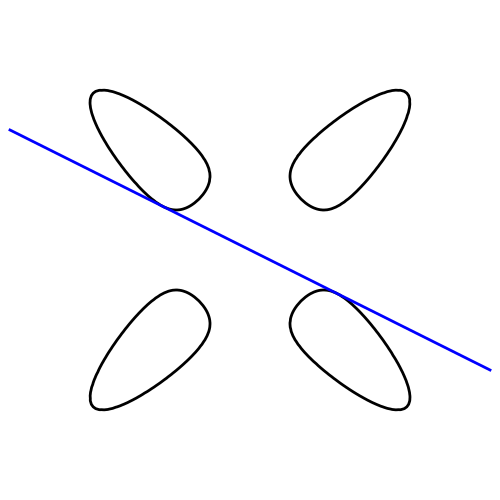}
\caption{In this affine chart the type of the displayed bitangents is $+1$ (left) and $-1$ (right), independent of the choice of $X_0$.}
\label{fig:bitaff}
\end{figure}
 
\begin{figure}[ht]
 \includegraphics[width=5cm]{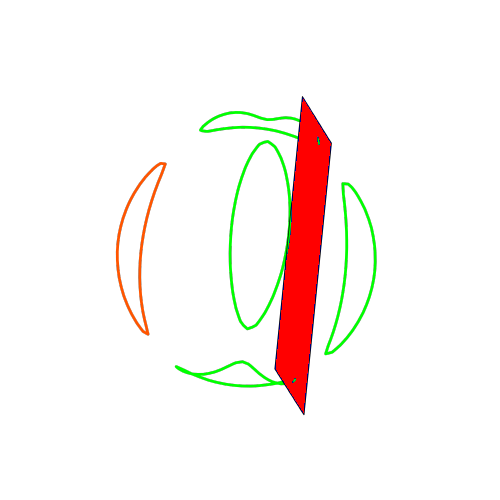} \quad
 \includegraphics[width=4cm]{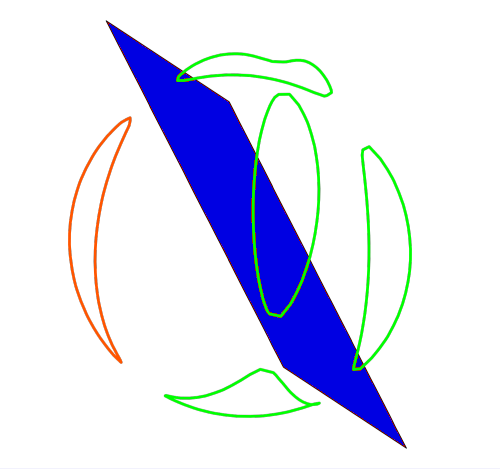}
\caption{Both figures display a genus four curve canonically embedded as a space sextic. The orange oval is $X_0$.
In this affine chart the type of the displayed tritangent planes is $+1$ (left) and $-1$ (right).}
\label{fig:tritaff}
\end{figure}

\begin{cor}\label{cor:cont3}
 Let $H_\infty$ a real hyperplane which does not interesect $X$ in real points. Then
 \begin{eqnarray*}
  \sum_{H\textrm{ real contact hyperplane}}\textnormal{Type}_{X_0}(H,H_\infty)=2^{g-1}.
 \end{eqnarray*}
\end{cor}

If the genus of $X$ is odd and $H$ is a contact hyperplane, then $\textrm{Type}_{X_0}(H,H_\infty)$ does not depend on the choice of $X_0$ by \Cref{rem:choiceA}. The case $g=3$ of the following corollary is precisely \cite[Theorem~1]{larsonvogt}.

\begin{cor}\label{cor:affineodd}
 Let $g$ be odd and let $H_\infty$ a real hyperplane which does not interesect $X$ in real points. Then
 \begin{eqnarray*}
  \sum_{H\textrm{ real contact hyperplane}}\textnormal{Type}(H,H_\infty)=2^{g-1}.
 \end{eqnarray*}
\end{cor}

\section{A conjectural arithmetic count}\label{sec:arith}

\subsection{The $\A^1$-degree}
We first recall some preliminaries from $\A^1$-enumerative geometry, mostly following the exposition of \cite[\S8]{pauliwickelgren}. Let $k$ always denote a field.
Recall that the \emph{Grothendieck--Witt group} $\GW(k)$ of $k$ is the Grothendieck group of the monoid of isometry classes of nonsingular quadratic spaces with addition given by orthogonal sum. The element of $\GW(k)$ defined by the quadratic form
\begin{equation*}
 k\to k,\, x\mapsto ax^2
\end{equation*}
for $a\in k^\times$ is denoted by $\langle a\rangle$. 
We consider $k^{\times}/ (k^\times)^2$ as a subset of $\GW(k)$ via the injection that sends the square class of $a\in k^{\times}$ to $\langle a\rangle$.
For a finite separable field extension $K/k$ there is a natural group homomorphism
\begin{equation*}
 \Tr_{K/k}\colon \GW(K)\to\GW(k)
\end{equation*}
defined by composing a quadratic form $V\to K$ on a $K$-vector space $V$ by the field trace $\tr_{K/k}\colon K\to k$.

\begin{Def}[Def.~7 in \cite{pauliwickelgren}]
Let $f\colon X\to Y$ be a finite surjective morphism of non-singular $k$-varieties. A \emph{relative orientation} of $f$ is an isomorphism $$L\otimes L\to\shHom(\det\cT_X, f^*\det\cT_Y)$$ where $L$ is a line bundle on $X$ and $\cT_X$ resp. $\cT_Y$ is the tangent bundle on $X$ and $Y$ respectively.
\end{Def}

Let $f\colon X\to Y$ be a finite surjective morphism of non-singular $K$-varieties, relatively orientated by the isomorphism $\psi\colon L\otimes L\to\shHom(\det\cT_X, f^*\det\cT_Y)$. Let $x\in X$ a closed point outside the ramification locus of $f$ such that $\kappa(x)/k$ is separable and let $y=f(x)\in Y$. We consider the induced map on tangent spaces $$\T_xf\colon \T_xX\to \T_yY\otimes_{\kappa(y)}\kappa(x).$$ After choosing bases that are compatible in the sense that the corresponding element of the fiber of $\shHom(\det\cT_X, f^*\det\cT_Y)$ at $x$ is the image of a square under $\psi$ \cite[Def.~8]{pauliwickelgren}, we take the determinant $\J_xf$ of $\T_xf$. It is straight-forward to check that the class $\langle \J_xf\rangle\in\GW(\kappa(x))$ does not depend on the chosen compatible bases. Then one defines the \emph{local $\A^1$-degree  of $f$ at $x$ with respect to the relative orientation $\psi$} as
\begin{equation*}
 \deg_x^{\A^1}(f,\psi):=\Tr_{\kappa(x)/\kappa(y)}\langle \J_xf\rangle\in\GW(\kappa(y)).
\end{equation*}
In our notation we keep track of the relative orientation because later on we will consider the degree of the same morphism with respect to different relative orientations.

\subsection{Definition of $b_2$ in terms of the $\A^1$-degree}\label{sec:orienta1}
We consider the identity map
\begin{equation*}
 \id\colon A\to A.
\end{equation*}
A relative orientation of $\id$ is an isomorphism $\psi\colon L\otimes L\to\cO_A$ for some $2$-torsion line bundle $L$ on $A$. Such an isomorphism is given by a rational function $f$ on $A$ such that $\divv(f)=2D$ where $D$ is a divisor whose class corresponds to the line bundle $L$. The local degree of $\id$ at a closed point $x\in A$ not in the support of $D$ with respect to this relative orientation is then just given by 
\begin{equation*}
 \deg_x^{\A^1}(\id,\psi)=\langle {f}(x)\rangle\in\GW(\kappa(x)).
\end{equation*}
We say that a relative orientation
\begin{equation*}
 \psi\colon L\otimes L\to\cO_A
\end{equation*}
is \emph{induced by $a'\in A^\vee_2(k)$} if $L$ is the line bundle corresponding to $a'$ and further $\deg_{0}^{\A^1}(\id,\psi)=\langle1\rangle$. If two relative orientations $\psi$ and $\psi'$ are induced by $a'$, then $\deg_{a}^{\A^1}(\id,\psi)=\deg_{a}^{\A^1}(\id,\psi')$ at all closed points $a\in A$ and we write
\begin{equation*}
 \deg_{a}^{\A^1}(\id,a'):=\deg_{a}^{\A^1}(\id,\psi).
\end{equation*}
With this notation, for all $a\in A_2(k)$ and $a'\in A^\vee_2(k)$, we have
 \begin{equation*}
  \deg_{a}^{\A^1}(\id,a')=b_2(a,a').
 \end{equation*}
 
\subsection{A conjectural arithmetic count} 
We make the following conjecture.

\begin{con}\label{con:main}
 Let $A$ be an abelian variety of dimension $g$ over a field $k$ of $\Char(k)\neq2$. Let $\lambda\colon A\to A^\vee$ a principal polarization. Then we have
 \begin{equation*}
  \sum_{a\in A_2}\Tr_{\kappa(a)/k}(\deg_{a}^{\A^1}(\id_{\kappa(a)},\lambda(a))) =2^{g-1}\cdot((2^g+1)\cdot\langle1\rangle+(2^g-1)\cdot\langle-1\rangle)
 \end{equation*}
 where $\id_{\kappa(a)}$ denotes the base-change of $\id\colon A\to A$ to $\kappa(a)$.
\end{con}

We conclude this section by proving the conjecture in a few special cases.

\begin{prop}
 \Cref{con:main} is true when $g=1$.
\end{prop}

\begin{proof}
 Let $(E,0)$ be an elliptic curve in Weierstrass normal form 
 \begin{equation*}
  y^2=p(x)
 \end{equation*}
 where $p$ is a monic polynomial of degree three. The principal polarization $\lambda$ of $E$ sends a point $a$ to the divisor class of $a-0$.
 Let $z_1,z_2,z_3\in\bar{k}$ the zeros of $p$. The non-trivial $2$-torsion points of $E(\bar{k})$ are given by $a_i=(z_i,0)$ for $i=1,2,3$. Further, if $\{i,j,k\}=\{1,2,3\}$ and $z\in k$ with $z\neq z_i$, then the rational function $f=\frac{(x-z_j)(x-z_k)}{(x-z)^2}$ is defined over $k(z_i)=\kappa(a_i)$ and $\frac{1}{2}\divv(f)$ is linearly equivalent to $a_i-0$. Thus we can compute $\deg_{a_i}^{\A^1}(\id_{\kappa(a_i)},\lambda(a_i))$ as the square class of
 \begin{equation*}
  \frac{(z_i-z_j)(z_i-z_k)}{(z_i-z)^2}=(z_i-z_j)(z_i-z_k)=p'(z_i)=\frac{1}{p'(z_i)}
 \end{equation*}
 where $p'$ is the derivative of $p$. Letting $R=k[x]/(p)$ we thus have that \begin{equation*} \sum_{0\neq a\in A_2}\Tr_{\kappa(a)/k}(\deg_{a}^{\A^1}(\id_{\kappa(a)},\lambda(a)))\end{equation*} equals the isometry class of the quadratic form \begin{equation*}R\to k,\, f\mapsto\tr_{A/k}\left(\frac{f^2}{p'}\right).\end{equation*} By \cite[Lemma III.6.2]{serreloc} this is $2\langle1\rangle+\langle-1\rangle$. Since the local degree at $0$ is $\langle1\rangle$ by definition, this implies the claim.
\end{proof}

\begin{lem}\label{lem:determinant}
Let $A$ be a principally polarized abelian variety of dimension $g$ over a field $k$ of $\Char(k)\neq2$.
Assume that all $2$-torsion points of $A$ are $k$-rational. Then
 \begin{equation*}
  \prod_{a\in A_2(k)} q_2(a)=(-1)^{2^{2(g-1)}}\in k^\times/(k^\times)^2.
 \end{equation*}
\end{lem}

\begin{proof}
 By \Cref{thm:q2prop} we have for all $c_1,\ldots,c_n\in A_2(k)$ that
 \begin{equation*}
  q_2(c_1+\ldots+c_n)=\prod_{1\leq i<j\leq n}e_2(c_i,c_j)\cdot\prod_{l=1}^n q_2(c_l)\in k^\times/(k^\times)^2.
 \end{equation*}
 Let $a_1,\ldots,a_g,b_1,\ldots,b_g$ be a symplectic basis of $A_2(k)$. Then:
\begin{equation*}
   \prod_{a\in A_2(k)} q_2(a)=\prod_{S\subset[g],T\subset[g]}(-1)^{|S|\cdot|T|}\left(\prod_{i\in S}q_2(a_i)\prod_{i\in T}q_2(b_i)\right)
\end{equation*}
where $[g]=\{1,\ldots,g\}$. Each $q_2(a_i)$ and each $q_2(b_i)$ appears $2^{2g-1}$ times in the product which is an even number. Thus the product equals
\begin{equation*}
 \prod_{S\subset[g],T\subset[g]}(-1)^{|S|\cdot|T|}.
\end{equation*}
The number of subsets of $[g]$ with an odd number of elements is $2^{g-1}$ which implies the claim.
\end{proof}

\begin{cor}\label{cor:det}
 Let $A$ be a principally polarized abelian variety of dimension $g$ over a field $k$ of $\Char(k)\neq2$.
Assume that all $2$-torsion points of $A$ are $k$-rational. If every binary form over $k$ is universal, then \Cref{con:main} is true for $A$.
\end{cor}

\begin{proof}
 By \cite[Theorem~II.3.5]{lam} two quadratic forms over $k$ are isometric if they have the same dimension and determinant. The determinants of both sides in the equality in \Cref{con:main} agree by \Cref{lem:determinant}.
\end{proof}

\begin{rem}
 The assumption in \Cref{cor:det} that every binary form over $k$ is universal is for example satisfied by every finite field and every nonreal field of transcendence degree $1$ over a real closed field \cite[Example~XI.6.2]{lam}.
\end{rem}

%\bigskip

 \noindent \textbf{Acknowledgements.}
I would like to thank Philip Dittmann for pointing me towards Weil's reciprocity law.

 \bibliographystyle{alpha}
 \bibliography{biblio}
 \end{document}